\documentclass[10pt]{amsart}
\usepackage[all]{xy}
\usepackage{amsmath}
\usepackage{amsfonts}
\usepackage{amssymb}
\usepackage{amscd}
\usepackage{amsthm}
\usepackage{latexsym}
\usepackage{amsbsy}
\usepackage{graphicx}
\newtheorem{lemma}{Lemma}[section]
\newtheorem{proposition}{Proposition}[section]
\newtheorem{theorem}{Theorem}[section]
\newtheorem{corollary}{Corollary}[section]
\newtheorem{definition}{Definition}[section]

\newtheorem{remark}{Remark}[section]

\usepackage{lipsum}
\def \Tr{{\rm Tr}}
\def \TR{{\rm TR}}
\def \Res{{\rm Res}}
\def \res{{\rm res}}
\def \A{\mathcal{A}_{\theta}}
\def\Int{\int\hspace{-0.35cm}{-} \,}

\title{The Curvature of the Determinant Line Bundle  on the  Noncommutative Two Torus}

\author{Ali Fathi, Asghar Ghorbanpour, Masoud Khalkhali} 
\date{}

\begin{document}

\maketitle

\begin{center} 

Department of Mathematics, The University of Western Ontario\\London, ON, Canada \footnote{{\it E-mail addresses}:  
afathiba@uwo.ca, aghorba@uwo.ca, masoud@uwo.ca}
 \\

\end{center}

\begin{abstract}
We compute the curvature of the determinant line bundle on  a family of Dirac operators for  a noncommutative two torus. Following Quillen's original construction for Riemann surfaces and using zeta regularized determinant of Laplacians,  one can endow the determinant line bundle  with a natural Hermitian metric. 
By using an analogue of  Kontsevich-Vishik  canonical trace,  defined on Connes' algebra of classical pseudodifferential symbols  for  the  noncommutative two torus, we compute the curvature form of the determinant line bundle by computing the second variation $\delta_{w}\delta_{\bar{w}}\log\det(\Delta)$. 
\end{abstract}

\tableofcontents

\section{Introduction}

In this paper we compute the curvature of the determinant line bundle associated to a family of Dirac operators on the noncommutative two torus. Following Quillen's pioneering work \cite{Quillen1985},  and using zeta regularized determinants,  one can endow the determinant line bundle over the space of Dirac operators on the noncommutative two torus with a natural Hermitian metric. Our result computes the curvature of the associated Chern connection on this holomorphic line bundle. In the noncommutative case the method of proof applied in \cite{Quillen1985} does not work and we had to use  a different strategy. To this end we found it very useful to extend the canonical trace of Kontsevich-Vishik \cite{Kontsevich-Vishik1995} to the  algebra of pseudodifferential operators on the  noncommutative two torus.

This paper is organized as follows. In Section 2 we review some standard  facts about  Quillen's determinant line bundle on the space of  Fredholm operators  from \cite{Quillen1985}, and  
  about  noncommutative two torus  that we need in this paper. 
In Section 3 we develop the tools that are needed in our   computation of the curvature of the determinant line bundle in the noncommutative case. 
We recall  Connes' pseudodifferential calculus and define  an analogue of the Kontsevich-Vishik  trace for  classical pseudodifferential symbols  on the noncommutative torus. 
A similar construction of the canonical trace can be found in \cite{Paycha-Levy2014}, where one works with the algebra of toroidal symbols. 
Section 4 is devoted to Cauchy-Riemann operators on $\A$ with a fixed complex structure.  This is the family of elliptic operators that we want to study its  determinant line bundle. 
In Section 5 using the  calculus of symbols and the canonical trace   we compute the curvature of determinant line bundle. Calculus of symbols and the canonical trace  allow us to bypass    local calculations  involving  Green functions   in \cite{Quillen1985}, which is  not applicable in our noncommutative case.   

The study of the conformal  and complex geometry of   the noncommutative two torus started with  the seminal  work  \cite{Connes-Tred2009} (cf. also  \cite{Connes-Cohen1992} for a preliminary version) where a Gauss-Bonnet theorem is proved for a  noncommutative two torus equipped with a conformally perturbed metric. 
This result was refined and extended in  \cite{Masoud-Farzad2012} where the Gauss-Bonnet theorem was  proved  for  metrics in all translation invariant conformal structures. 
 The  problem of  computing the scalar curvature of  the curved noncommutative two 
torus was fully settled  in \cite{Connes-Moscovici2014}, and, independently,  in \cite{Farzad-Masoud2013},  and in \cite{Farzad-Masoud2014} for the four dimensional case.  Other related works  include
\cite{Marcolli-Tanvir2012, Ali-Masoud2014, dabsit1,  dabsit2, Ros, Lesch2}.

\section{Preliminaries}
In this section we recall the definition of  Quillen's  determinant line bundle  over the space of Fredholm operators.  We also recall some basic notions about  noncommutative torus that we need in this paper.

\subsection{The determinant line bundle}

Unless otherwise  stated, in this paper by a Hilbert space  we mean a separable infinite dimensional Hilbert space over the field of complex numbers. Let $ \mathcal{F} = {\rm Fred} (\mathcal{H}_0, \mathcal{H}_1)$  denote the set of Fredholm operators between Hilbert spaces $\mathcal{H}_0$ and $ \mathcal{H}_1$. It is an open subset, with respect to norm topology,  in the complex Banach space of all bounded linear operators  between  $\mathcal{H}_0$ and $ \mathcal{H}_1$. The index map $ index : \mathcal{F}  \to \mathbb{Z}$ is a homotopy invariant and in fact defines a bijection between connected components of 
$\mathcal{F}$ and  the set of integers $\mathbb{Z}$.

  It is well known that $\mathcal{F}$ is a classifying space for $K$-theory: for any compact space $X$ we have a natural ring  isomorphism 
$$ K^0 (X) = [ X, \mathcal{F}]$$
between the $K$-theory of $X$ and the set of homotopy classes of continuous maps from $X$ to 
$\mathcal{F}.$ In other words,   homotopy classes of continuous families of Fredholm operators parametrized by $X$ determine the $K$-theory of $X$. 
It thus follows that  $\mathcal{F}$ is homotopy equivalent to $\mathbb{Z}  \times BU$, the latter being also a classifying space  for $K$-theory. Let   $\mathcal{F}_0$ denote the set of Fredholm operators   with index zero.  By 
 Bott periodicity, $\pi_{2j}(\mathcal{F})\cong \mathbb{Z}$ and  $\pi_{2j+1}(\mathcal{F}) = \{0\}$ for $j\ge 0$.
So by Hurewicz's theorem, $H^2(\mathcal{F}_0, Z) \cong \mathbb{Z}.$ Now the determinant line bundle ${\rm DET}$ defined below has the property that its  first Chern class,
$c_1({\rm DET}),$ is  a  generator of $H^2(\mathcal{F}_0, \mathbb{Z}) \cong \mathbb{Z}$. We refer to \cite{Scott2010, Chakraborty-Mathai2009}  and references therein for details.

In \cite{Quillen1985} Quillen defines a line bundle $\text{DET} \to \mathcal{F}$ such that for any $T \in  \mathcal{F}$
$${\rm DET}_T = \Lambda^{max}({\rm ker}(T))^* \otimes  \Lambda^{max}({\rm coker}(T)).$$
This is remarkable if we notice that ${\rm ker}(T)$ and ${\rm coker}(T)$ are not vector bundles due to  
discontinuities in their dimensions as $T$ varies within $\mathcal{F}$. Let us briefly recall the construction of 
this determinant line bundle DET. For each finite dimensional subspace $F$ of $ \mathcal{H}_1$ let 
$U_F = \{ T\in \mathcal{F}_1: {\rm Im}(T) + F = \mathcal{H}_1\}$ denote the set of Fredholm operators whose range is transversal to $F$. It is an open subset of $\mathcal{F}$ and we have an open cover
$ \mathcal{F} = \bigcup U_F$.

For $T\in U_F$,  the exact
sequence 
\begin{equation}\label{eqn:index1}
0\to {\rm ker}(T) \to T^{-1}F \stackrel{T}{\to} F \to {\rm coker}(T)\to 0
\end{equation}
shows that  the rank of
$T^{-1}F$ is constant when $T$ varies within a continuous family in  $  U_F$.
Thus we can define a  vector bundle  $\mathcal{E}^F\to U_F$
 by setting $\mathcal{E}^F_T = T^{-1}F.$ We can then define a line bundle  ${\rm DET}^F \to U_F$ by setting 
$$ {\rm DET}^F_T= \Lambda^{max} (T^{-1}F)^* \otimes   \Lambda^{max} F.$$
We can use the inner products on $\mathcal{H}_0$ and $\mathcal{H}_1$ to split the above  exact sequence  \eqref{eqn:index1} canonically  and get a canonical isomorphism 
$ {\rm ker}(T) \oplus F \cong T^{-1}F \oplus {\rm coker}(T)$.  Therefore
$$
\Lambda^{max}({\rm ker}(T))^* \otimes  \Lambda^{max}({\rm coker}(T))\cong
\Lambda^{max} (T^{-1}F)^* \otimes   \Lambda^{max} F.
$$
Now over each member of the cover  $U_F$   a line bundle
 ${\rm DET}^F \to U_F$ is defined.  Next one shows that over intersections $U_{F_1} \cap U_{F_2}$ there is an isomorphism ${\rm DET}^{F_1} \to {\rm DET}^{F_2} $  and moreover the isomorphisms satisfy a 
 cocycle condition  over triple intersections $U_{F_1} \cap U_{F_2} \cap U_{F_3}.$ This shows that the line bundles ${\rm DET}^F \to U_F$ glue together to define a line bundle over $\mathcal{F}$. It is further shown in \cite{Quillen1985} that this line bundle is holomorphic  as a bundle over an open subset of a complex  Banach space.

 It is tempting to think that since $c_1 (\text{DET})$ is the generator of $H^2(\mathcal{F}_0, \mathbb{Z}) \cong \mathbb{Z},$ there might exits  a natural Hermitian metric on DET  whose curvature 2-form would be a representative of this   generator. One problem is that the induced metric from   ${\rm ker}(T)$ and ${\rm ker}(T^*)$  on  DET is not even continuous.  
 In \cite{Quillen1985} Quillen shows that  for  families  of Cauchy-Riemann operators  on a Riemann surface one can correct the Hilbert space metric by  multiplying it by   zeta regularized determinant and in this way one obtains a smooth Hermitian metric on the induced determinant line bundle. In  Section 5 we  describe a similar construction  for noncommutative two torus.

\subsection{Noncommutative two torus}

 For  $\theta \in \mathbb{R}, $  the noncommutative two torus $A_{\theta}$ is   by definition
the universal unital $C^*$-algebra  generated by two unitaries $U, V$ satisfying
 \[VU=e^{2 \pi i \theta} UV.\]

There is a continuous action of $\mathbb{T}^2$, $\mathbb{T}= \mathbb{R}/2\pi \mathbb{Z}$, on $A_{\theta}$ by $C^*$-algebra
automorphisms  $\{ \alpha_s\}$, $s\in \mathbb{R}^2$, defined by
\[\alpha_s(U^mV^n)=e^{is.(m,n)}U^mV^n.\]
The space of smooth elements for this action will be denoted by
$A_{\theta}^{\infty}$. It is a dense subalgebra of $A_{\theta}$  which can be alternatively
described as the algebra of elements in $A_{\theta}$
whose (noncommutative) Fourier expansion has rapidly decreasing coefficients:
\[A_{\theta}^{\infty}=\left\{\sum_{m,n\in \mathbb{Z}}a_{m,n}U^mV^n: a_{m,n}\in\mathcal{S}(\mathbb{Z}^2)\right\}.\]
There is a  normalized, faithful and positive,  trace   $\varphi_0$  on $A_{\theta}$ whose restriction on smooth elements is given by
\[\varphi_0(\sum_{m,n\in \mathbb{Z}}a_{m,n}U^mV^n)=a_{0,0}.\]

 The algebra $A_{\theta}^{\infty}$   is equipped with the 
 derivations  $\delta_1, \, \delta_2: A_{\theta}^{\infty} \to A_{\theta}^{\infty}$, uniquely  defined by the relations 
\[\delta_1(U)=U, \,\, \delta_1(V)=0, \quad  \delta_2(U)=0, \,\, \delta_2(V)=V.\]
We have $\delta_j(a^*)= -\delta_j(a)^* $ for $j=1, 2$ and all $a\in A_{\theta}^{\infty}$.
 Moreover, the 
 analogue of the
integration by parts formula in this setting is given by:
\[ \varphi_0(a\delta_j(b)) = -\varphi_0(\delta_j(a)b), \,\,\, \forall a,b \in A_{\theta}^{\infty}. \]

We apply the GNS construction to $A_{\theta}$.  The state $\varphi_0$  defines an  inner product 
\[ \langle a, b \rangle = \varphi_0(b^*a), \,\,\, a,b \in A_{\theta}, \nonumber \]
and a pre-Hilbert structure on $A_{\theta}$. After completion we obtain  a  Hilbert space denoted 
$\mathcal{H}_0$. The derivations $\delta_1, \delta_2$, as densely defined unbounded
operators on $\mathcal{H}_0$,  are formally selfadjoint and have unique extensions to selfadjoint operators.

We introduce a complex structure associated with a complex number $\tau = \tau_1+i\tau_2, \, \tau_2 >0,$
by defining
\[ \bar{\partial} = \delta_1 + \tau \delta_2, \,\,\, \bar{\partial}^*=  \delta_1 + \overline{\tau} \delta_2. \]


Note that   $\bar{\partial}$ is an unbounded operator on  $\mathcal{H}_0$ and $\bar{\partial}^*$ is
 its formal adjoint. The analogue of the space of anti-holomorphic 1-forms on the ordinary two torus
is defined to be  $$\Omega^{0,1}_{\theta}=\left\{\sum a \bar{\partial} b\;, a,b
\in A_{\theta}^{\infty}\right\}.$$
Using the induced inner product from $\psi $, one can turn $\Omega^{0,1}_{\theta}$ into a Hilbert space which we denote by
$\mathcal{H}^{0,1}$.

\section{The canonical trace and noncommutative residue}
In this section we define an analogue of the canonical trace of Kontsevich and Vishik \cite{Kontsevich-Vishik1995} for the  noncommutative torus. Let us first recall the algebra of pseudodifferential symbols on the noncommutative torus \cite{Connes1980, Connes-Tred2009}.

 \subsection{Pseudodifferential calculus on $\A$}

  Using operator valued symbols, one can define an algebra of  pseudodifferential operators on  $A_{\theta}^{\infty}$. We shall use the notation
$\partial^{\alpha}=\frac{\partial^{\alpha_1}}{\partial \xi_1^{\alpha_1}}  \frac{\partial^{\alpha_2}}{\partial \xi_2^{\alpha_2}} $,  and $\delta^{\alpha}= \delta_1^{\alpha_1}\delta_2^{\alpha_2}, $  for a multi-index $\alpha= (\alpha_1, \alpha_2).$

\begin{definition}
For a real number $m$, a smooth map $\sigma: \mathbb{R}^2 \to A_{\theta}^{\infty}$ is  said
to be a symbol of order $m$, if for all non-negative integers $i_1, i_2, j_1,
j_2,$
\[ ||\delta^{(i_1, i_2)}  \partial^{(j_1, j_2)}  \sigma(\xi) ||
\leq c (1+|\xi|)^{m-j_1-j_2},\]
where $c$ is a constant, and if there exists a smooth map $k: \mathbb{R}^2 \to
A_{\theta}^{\infty}$ such that
\[\lim_{\lambda \to \infty} \lambda^{-m} \sigma(\lambda\xi_1, \lambda\xi_2) = k (\xi_1, \xi_2).\]
The space of symbols of order $m$ is denoted by ${\mathcal S}^m(\A)$.
\end{definition}
\begin{definition}\label{pseudodef}
To a symbol $\sigma$ of order $m$, one can associate an operator on $A_{\theta}^{\infty}$,
denoted by $P_{\sigma}$, given by
$$ P_{\sigma}(a) = \int \int e^{-is \cdot \xi} \sigma(\xi) \alpha_s(a) \,ds \,
d\xi.$$
Here, $d\xi=(2\pi)^{-2}d_L\xi$ where $d_L\xi$ is the Lebesgue measure on $\mathbb{R}^2$.
The operator $P_{\sigma}$ is said to be a pseudodifferential operator of order $m$.
\end{definition}
For
example, the differential operator $\sum_{j_1+j_2 \leq m } a_{j_1,j_2}
 \delta^{(j_1, j_2)} $ is associated with the symbol $\sum_{j_1+j_2 \leq m } a_{j_1,j_2}
\xi_1^{j_1} \xi_2^{j_2}$ via the above formula.

Two symbols $\sigma$, $\sigma'\in {\mathcal S}^m(\A)$  are said to be equivalent if and only if $\sigma-\sigma'\in {\mathcal S}^n(\A)$ for all integers $n$. The equivalence of the symbols will be denoted by  $\sigma \sim \sigma'$.

Let $P$ and $Q$ be pseudodifferential operators with the symbols
$\sigma$ and $\sigma'$ respectively. Then the adjoint $P^*$ and
the product $PQ$ are pseudodifferential operators with the following
symbols
\[
\sigma(P^*) \sim \sum_{ \ell = (\ell_1, \ell_2) \geq 0} \frac{1}{\ell ! }
\partial^{\ell} \delta^{\ell}
(\sigma(\xi))^*,
\]

\[
\sigma (P Q) \sim \sum_{\ell = (\ell_1, \ell_2) \geq 0} \frac{1}{\ell !}
\partial^{\ell} (\sigma (\xi))
\delta^{\ell}(\sigma'(\xi)).
\]

\begin{definition}

A symbol $\sigma\in\mathcal{S}^m(\mathcal{A}_{\theta})$ is called elliptic if $\sigma(\xi)$ is invertible for $\xi\neq0,$ and for some $c$
$$||\sigma(\xi)^{-1}||\leq c(1+|\xi|)^{-m},$$ for large enough $|\xi|$.
\end{definition}

A  smooth map $\sigma:\mathbb{R}^2\to \A$ is called a classical symbol of order $\alpha \in \mathbb{C}$  
if for any $N$ and each $0\leq j\leq N$ there exist $\sigma_{\alpha-j}:\mathbb{R}^2\backslash\{0\}\to\mathcal{A}_{\theta}$ positive homogeneous of degree $\alpha-j$,  and a symbol $\sigma^N\in\mathcal{S}^{\Re(\alpha)-N-1}(\mathcal{A}_{\theta})$, such that
 \begin{equation}
 \label{sigma}
\sigma (\xi)=\sum_{j=0}^{N}\chi(\xi)\sigma_{\alpha-j}(\xi)+\sigma^N(\xi)\quad\xi\in\mathbb{R}^2.
 \end{equation}
 Here  $\chi$ is a smooth cut off function on $\mathbb{R}^2$ which is equal to zero on a small ball around the origin,  and is equal to one outside the unit ball.
  It can be shown that the homogeneous terms in the expansion are uniquely determined by $\sigma$. 
 We denote the set of classical symbols of order $\alpha$ by $\mathcal{S}^{\alpha}_{cl}(\mathcal{A}_{\theta})$ and the associated classical pseudodifferential operators by $\Psi_{cl}^{\alpha}(\mathcal{A}_{\theta})$.

 The space of classical symbols $\mathcal{S}_{cl}(\mathcal{A}_{\theta})$ is equipped with a Fr\'{e}chet topology induced by the semi-norms
\begin{equation}\label{frechet}
p_{\alpha,\beta}(\sigma)=\sup_{\xi\in\mathbb{R}^2}(1+|\xi|)^{-m+|\beta|}||\delta^{\alpha}\partial^{\beta}\sigma(\xi)||.
\end{equation}

The analogue of the Wodzicki residue for classical pseudodifferential operators on the noncommutative torus is defined in \cite{Farzad-Wong2011}.
 \begin{definition}
 The Wodzicki residue of a classical pseudodifferential operator $P_{\sigma}$ is defined as
 $$\Res(P_{\sigma})=\varphi_0\left(\res (P_\sigma)\right),$$
where $\res(P_\sigma):=\int_{|\xi|=1}\sigma_{-2}(\xi)d\xi$. 
 \end{definition}
It is evident from its  definition that Wodzicki residue vanishes on differential operators and on non-integer order classical pseudodifferential operators.
 \subsection{The  canonical trace} 
 In what follows, we define the analogue of Kontsevich-Vishik trace \cite{Kontsevich-Vishik1995}  on non-integer order pseudodifferential operators on the noncommutative torus. For an alternative approach based on toroidal noncommutative symbols see \cite{Paycha-Levy2014}. For a thorough review of the theory in the classical case we refer to \cite{Paycha2012,Paycha-Scott2007}. First we show the existence of the so called cut-off integral for classical symbols.
 
\begin{proposition}\label{expansioninR}
 Let $\sigma\in\mathcal{S}_{cl}^{\alpha}(\mathcal{A}_{\theta})$ and $B(R)$ be the ball of radius $R$ around the origin. One has the following asymptotic expansion
 $$\int_{B(R)}\sigma(\xi)d\xi\sim_{R\rightarrow\infty}\sum_{j=0,\alpha-j+2\neq0}^{\infty}\alpha_j(\sigma)R^{\alpha-j+2}+\beta(\sigma)\log R+c(\sigma),$$
 where $\beta(\sigma)=\int_{|\xi|=1}\sigma_{-2}(\xi)d\xi$
 and the constant term in the expansion, $c(\sigma)$, is given by 
\begin{equation}\label{cutoffinexpansion}
\int_{\mathbb{R}^n}\sigma^N+\sum_{j=0}^{N}\int_{B(1)}\chi(\xi)\sigma_{\alpha-j}(\xi)d\xi-\sum_{j=0,\alpha-j+2\neq0}^N \frac{1}{\alpha-j+2}\int_{|\xi|=1}\sigma_{\alpha-j}(\omega)d\omega.
\end{equation}
 Here we have used the notation of (\ref{sigma}).
  \end{proposition}
 \begin{proof}
  First, we write $\sigma (\xi)=\sum_{j=0}^{N}\chi(\xi)\sigma_{\alpha-j}(\xi)+\sigma^N(\xi)$
with large enough $N$, so that $\sigma^N$ is integrable. Then we have,
\begin{equation}\label{constant1}
\int_{B(R)}\sigma(\xi)d\xi=\sum_{j=0}^{N}\int_{B(R)}\chi(\xi)\sigma_{\alpha-j}(\xi)d\xi+\int_{B(R)}\sigma^N(\xi)d\xi.
\end{equation}
For $N>\alpha+1$, $\sigma^N\in\mathcal{L}^1(\mathbb{R}^2,\mathcal{A}_{\theta})$, so
\begin{equation}
\int_{B(R)}\sigma^N(\xi)d\xi\to\int_{\mathbb{R}^2}\sigma^N(\xi)d\xi,\quad R\to\infty.\nonumber
\end{equation}
Now for each $j\leq N$ we have
$$\int_{B(R)}\chi(\xi)\sigma_{\alpha-j}(\xi)d\xi=\int_{B(1)}\chi(\xi)\sigma_{\alpha-j}(\xi)d\xi+\int_{B(R)\backslash B(1)}\chi(\xi)\sigma_{\alpha-j}(\xi)d\xi.$$
Obviously $\int_{B(1)}\chi(\xi)\sigma_{\alpha-j}(\xi)d\xi<\infty$ and 
by using polar coordinates $\xi=r\omega $, and homogeneity of $\sigma_{\alpha-j}$, we have
\begin{equation}\label{constant2}
\int_{B(R)\backslash B(1)}\chi(\xi)\sigma_{\alpha-j}(\xi)d\xi=\int_{1} ^R r^{\alpha-j+2-1}dr\int_{|\xi|=1}\sigma_{\alpha-j}(\xi)d\xi.
\end{equation}
Note that the cut-off function is equal to one on the set $\mathbb{R}^2 \backslash B(1)$. 
For the term with  $\alpha-j=-2$ one has
\begin{equation}
\int_{B(R)\backslash B(1)}\chi(\xi)\sigma_{\alpha-j}(\xi)d\xi= \log R  \int_{|\xi|=1}\sigma_{\alpha-j}(\xi)d\xi.\nonumber
\end{equation}
The terms with $\alpha-j\neq -2$ will give us the following:
\begin{eqnarray}\label{constant3}
\int_{ B(R)\backslash B(1)}\chi(\xi)\sigma_{\alpha-j}(\xi)d\xi&=& \\
\frac{R^{\alpha-j+2}}{m-j+2}\int_{|\xi|=1}\sigma_{\alpha-j}(\xi)d\xi&-&\frac{1}{\alpha-j+2}\int_{|\xi|=1}\sigma_{\alpha-j}(\xi)d\xi.\nonumber
\end{eqnarray}
Adding all the constant terms in  \eqref{constant1}-\eqref{constant3}, we get the constant term given in \eqref{cutoffinexpansion}.
\end{proof}
\begin{definition} 
 The cut-off integral of a symbol $\sigma\in\mathcal{S}_{cl}^{\alpha}(\mathcal{A}_{\theta})$ is defined to be the constant term in the above asymptotic expansion, and we denote it by $\Int \sigma(\xi)d\xi$.
\end{definition}
  \begin{remark} Two remarks are in order here. First note that the cut-off integral of a symbol is independent of  the choice of $N$. Second, it is also independent of the choice of the cut-off function 
  $\chi$.   
    \end{remark}
 
 We now give the definition of the canonical trace for classical pseudodifferential operators on $\A$.
 \begin{definition}
 The canonical trace of a classical pseudodifferential operator $P\in\Psi^{\alpha}_{cl}(\mathcal{A}_{\theta})$ of non-integral order $\alpha$ is defined as 
 \begin{equation*}
 \TR(P):=\varphi_0\left(\Int \sigma_P(\xi)d\xi\right).
 \end{equation*}
 \end{definition}
 In the following, we establish the relation between the TR-functional and the usual trace on trace-class pseudodifferential operators.
Note that any pseudodifferential operator $P$  of order less that  $-2$, is a trace-class operator on $\mathcal{H}_0$ and its trace is given by
 $$\Tr(P)=\varphi_0\left(\int_{\mathbb{R}^2}\sigma_{P}(\xi)d\xi\right).$$
 On the other hand, for such operator the symbol is integrable and we have
\begin{equation}\label{TRTr}
\Int\sigma_P(\xi)=\int_{\mathbb{R}^2}\sigma_P(\xi)d\xi.
\end{equation}
 Therefore, the $\TR$-functional and operator trace coincide on classical pseudodifferential operators of order less than $-2$. 
 
Next, we show that the $\TR$-functional is in fact an analytic continuation of the operator trace and using this fact  we can prove that it is actually a trace.

 \begin{definition}
A family of symbols $\sigma(z)\in\mathcal{S}_{cl}^{\alpha(z)}(\A)$, parametrized by $z\in W\subset \mathbb{C}$, is called a holomorphic family if
\begin{itemize}
\item[i)]  The map $z\mapsto \alpha(z)$ is holomorphic. 
\item[ii)]  The map $z\mapsto \sigma(z)\in \mathcal{S}_{cl}^{\alpha(z)}(\mathcal{A}_{\theta})$ is a holomorphic map from $W$ to the Fr\'{e}chet space $\mathcal{S}_{cl}(\mathcal{A}^n_{\theta}).$  
\item[iii)] The map $z\mapsto \sigma(z)_{\alpha(z)-j}$ is   holomorphic for any $j$, where 
 \begin{equation}\label{sigma2}
 \sigma(z)(\xi)\sim\sum_j\chi(\xi)\sigma(z)_{\alpha(z)-j}(\xi)\in\mathcal{S}_{cl}^{\alpha(z)}(\mathcal{A}_{\theta}).
 \end{equation}
\item[iv)] The bounds of the asymptotic expansion of  $\sigma(z)$ are locally uniform with respect to $z$, i.e, for any $N\geq1$ and compact subset $K\subset W$, there exists a constant $C_{N,K,\alpha,\beta}$ such that for all multi-indices  $\alpha,\beta$ we have
$$\left|\left|\delta^{\alpha}\partial^{\beta}\left(\sigma(z)-\sum_{j<N}\chi\sigma(z)_{\alpha(z)-j}\right)(\xi)\right|\right|<C_{N,K,\alpha,\beta}|\xi|^{\Re(\alpha(z))-N-|\beta|}.$$
\end{itemize}

   A family $\left\{P_z\right\}\in\Psi_{cl}(\mathcal{A}_{\theta})$ is called holomorphic if $P_z=P_{\sigma(z)}$ for a holomorphic family of symbols $\left\{\sigma(z)\right\}$. 
  \end{definition}
 The following Proposition is an analogue of a result of Kontsevich and Vishik\cite{Kontsevich-Vishik1995}, for pseudodifferential calculus on noncommutative tori.
 \begin{proposition}\label{laurentofholo}
 Given a holomorphic family $\sigma(z)\in\mathcal{S}_{cl}^{\alpha(z)}(\mathcal{A}_{\theta})$, $z\in W \subset\mathbb{C}$, the map 
 $$z\mapsto \Int\sigma(z)(\xi)d\xi,$$
 is meromorphic with at most simple poles located in 
 $$P=\left\{z_0\in W;~\alpha(z_0)\in\mathbb{Z}\cap[-2,+\infty]\right\}.$$
 The residues at poles are given by
 $$\Res_{z=z_0} \Int\sigma(z)(\xi)d\xi=-\frac{1}{\alpha'(z_0)}\int_{|\xi|=1}\sigma(z_0)_{-2}d\xi.$$
 \end{proposition} 
 \begin{proof}
 By definition, one can write $\sigma(z)=\sum_{j=0}^N\chi(\xi)\sigma(z)_{\alpha(z)-i}(\xi)+\sigma(z)^N(\xi)$, and by Proposition \ref{expansioninR} we have,
\begin{align*}
\Int\sigma(z)(\xi)d\xi
&=\int_{\mathbb{R}^2}\sigma(z)^N(\xi)d\xi
+\sum_{j=0}^{N}\int_{B(1)}\chi(\xi)\sigma(z)_{\alpha(z)-j}(\xi)\\
&-\sum_{j=0}^{N}\frac{1}{\alpha(z)+2-j}\int_{|\xi|=1}\sigma(z)_{\alpha(z)-j}(\xi)d\xi.
\end{align*} 
 Now suppose $\alpha(z_0)+2-j_0=0$. By holomorphicity of $\sigma(z)$, we have $\alpha(z)-\alpha(z_0)=\alpha'(z_0)(z-z_0)+o(z-z_0)$. Hence
 $$\Res_{z=z_0} \Int\sigma(z)=\frac{-1}{\alpha'(z_0)}\int_{|\xi|=1}\sigma(z_0)_{-2}(\xi)d\xi.$$
 \end{proof}
 \begin{corollary}\label{analyticcontin}
 The functional $\TR$ is the analytic continuation of the ordinary trace on trace-class pseudodifferential operators.
 \end{corollary}
 \begin{proof}
 First observe that, by the above result, for a non-integer order holomorphic family of symbols $\sigma(z)$, the map $z\mapsto\Int\sigma(z)(\xi)d\xi$ is holomorphic.  Hence, the map $\sigma\mapsto\Int\sigma(\xi)d\xi$ is the unique analytic continuation of the map $\sigma\mapsto\int_{\mathbb{R}^2}\sigma(\xi)d\xi$ from $\mathcal{S}_{cl}^{<-2}(\mathcal{A}_{\theta})$ to $\mathcal{S}^{\notin\mathbb{Z}}_{cl}(\mathcal{A}_{\theta})$.  By \eqref{TRTr} we have the result.
\end{proof}

Let $Q\in\Psi^q_{cl}(\mathcal{A}_{\theta})$ be a positive elliptic pseudodifferential operator of order $q>0$. 
 The complex power of such an operator, $Q_\phi^z$, for $\Re(z)<0$  can be defined by the following Cauchy integral formula.
\begin{equation}\label{Qz}
Q_\phi^z=\frac{i}{2\pi}\int_{C_\phi}\lambda_\phi^z (Q-\lambda)^{-1} d\lambda.
\end{equation}
Here $\lambda^z_\phi$ is the complex power with branch cut $L_\phi =\{re^{i\phi}, r\geq 0\}$ and $C_\phi$ is a contour around the spectrum of $Q$ such that
$$C_\phi\cap {\rm spec}(Q)\backslash\{0\}=\emptyset,\qquad L_\phi\cap C_\phi=\emptyset,$$
$$ C_\phi\cap\{ {\rm spec}(\sigma(Q)^L(\xi)),\,\xi\neq 0\}=\emptyset.$$
In general an operator for which one can find a ray $L_\phi$ with the above property, is called an admissible operator with the spectral cut $L_\phi$. Positive elliptic operators are admissible and we take the ray $L_\pi$ as the spectral cut, and in this case we drop the index $\phi$ and write $Q^z$.

To extend \eqref{Qz} to $\Re(z)>0$ we choose a positive integer such that $\Re (z) <k$ and define 
$$Q_\phi^z:=Q^kQ_\phi^{z-k}.$$
It can be proved that this definition is independent of the choice of $k$.

\begin{corollary}
 Let $A\in\Psi ^{\alpha}_{cl}(\mathcal{A}_{\theta})$ be of order $\alpha\in\mathbb{Z}$ and let $Q$ be  a positive elliptic classical pseudodifferential operator of positive order $q$. We have
 $$\Res_{z=0}\TR(AQ^{-z})= \frac{1}{q}\Res(A).$$
 \end{corollary}
 \begin{proof}
 For the holomorphic family $\sigma(z)=\sigma(AQ^{-z})$,  $z=0$ is a pole for the map  $z\mapsto\Int\sigma(z)(\xi)d\xi$ whose  residue is given by 
 $$\Res_{z=0}\left(z\mapsto \Int\sigma(z)(\xi)d\xi\right)=-\frac{1}{\alpha'(0)}\int_{|\xi|=1}\sigma_{-2}(0)d\xi=-\frac{1}{\alpha'(0)}\res(A).$$
Taking trace on both sides gives the result. 
 \end{proof}
 Now we can prove the trace property of $\TR$-functional.
 \begin{proposition}
 We have  
 $\TR(AB)=\TR(BA)$ for any $A,B\in\Psi_{cl}(\mathcal{A}_{\theta})$, provided that $ord(A)+ord(B)\notin\mathbb{Z}$.
 \end{proposition}
\begin{proof}
 Consider the families $A_z=AQ^z$ and $B_z=BQ^z$  where $Q$ is an injective positive elliptic classical operator of order $q>0$.
 For $\Re(z)<<0$, the two families are trace class and $\Tr(A_zB_z)=\Tr(B_zA_z)$. By the uniqueness of the analytic continuation, we have
$$\TR(A_zB_z)=\TR(B_zA_z),$$ for those $z$ for which $2qz+{\rm ord}(A)+{\rm ord}(B)\not \in \mathbb{Z}.$ 
At $z=0$, we obtain $\Tr(AB)=\TR(BA).$
 \end{proof}
 
 \subsection{Log-polyhomogeneous symbols}  
 Proposition \ref{laurentofholo} can be extended and one can explicitly write down the Laurent expansion of the cut-off integral around each of the poles. The terms of the Laurent expansion involve residue densities of  $z$-derivatives of the holomorphic family. In general, $z$-derivatives of a classical holomorphic family of symbols is not classical anymore and therefore we introduce log-polyhomogeneous symbols which include the $z$-derivatives of the symbols of the holomorphic family $\sigma(AQ^{-z})$. 

\begin{definition}
A symbol $\sigma$ is called a log-polyhomogeneous symbol if it has the following form
\begin{equation}\label{logpolysym}
\sigma(\xi) \sim \sum_{j\geq 0}\sum_{l=0}^\infty\sigma_{\alpha-j,l}(\xi)\log^l|\xi|\quad |\xi|>0,
\end{equation}
with $ \sigma_{\alpha-j,l}$ positively homogeneous in $\xi$ of degree $\alpha - j$.
\end{definition}

 An important example  of an operator with such a symbol is $\log Q$ where $Q\in\Psi^q_{cl}(\mathcal{A}_{\theta})$ is a positive elliptic pseudodifferential operator of order $q>0$.
 The logarithm of $Q$ can be  defined by
 $$\log Q
 =Q\left.\frac{d}{dz}\right|_{z=0}Q^{z-1}
 =Q\left.\frac{d}{dz}\right|_{z=0}\frac{i}{2\pi}\int_C\lambda^{z-1}(Q-\lambda)^{-1}d\lambda.$$ 
 It is a pseudodifferential operator with symbol 
\begin{equation}\label{symboflog}\sigma(\log Q)
\sim \sigma(Q)\star \sigma\Big(\left.\frac{d}{dz}\right|_{z=0} Q^{z-1}\Big),
\end{equation}
where $\star$ denotes the products of the pseudodifferential symbols.  
 Using symbol calculus and homogeneity properties, we can show that \eqref{symboflog} is a log-homogeneous symbol of the form 
 $$\sigma(\log Q)(\xi)=2\log|\xi| I+ \sigma_{cl}(\log Q)(\xi),$$ 
 where  $\sigma_{cl}(\log Q)$ is a classical symbol of order zero. This symbol can be computed using the homogeneous parts of the classical symbol $\sigma(Q^z)=\sum_{j=0}^\infty b(z)_{2z-j}(\xi)$ and it is given by the following formula (see e.g. \cite{Paycha2012}).
  \begin{align}\label{cllog} \sigma_{cl}(\log Q)(\xi) &= \\ 
\sum_{k=0}^\infty \sum_{i+j+|\alpha|=k}\frac{1}{\alpha !}\partial^\alpha \sigma_{2-i}(Q)\delta^\alpha
& \left[|\xi|^{-2-j}\left.\frac{d}{dz}\right|_{z=0} b(z-1)_{2z-2-j}\left(\xi/|\xi|\right)\right]. \nonumber
  \end{align}
 
 The Wodzicki residue can also be extended to this class of pseudodifferential operators \cite{Lesch1999}. For an operator  $A$ with log-polyhomogeneous symbol as \eqref{logpolysym} it can be defined by
 $$\res(A)=\int_{|\xi|=1}\sigma_{-2,0}(\xi)d\xi.$$

By adapting the proof of Theorem 1.13 in \cite{Paycha-Scott2007} to the  noncommutative case, we have the following theorem which is written only for the families of the form $\sigma(AQ^{-z})$ which we  will use in Section \ref{sec:computation}. 
 
 \begin{proposition}\label{Laurentat0}
Let $A\in\Psi_{cl}^\alpha(\mathcal{A}_{\theta})$ and $Q$ be a positive , in general an admissible, elliptic pseudodifferential operator of positive order $q$. If $\alpha\in P$ then $0$ is a possible simple pole for the function $z\mapsto \TR(AQ^{-z})$ with the following Laurent expansion around zero. 
\begin{align*}
\TR(AQ^{-z})&=\frac{1}{q}\Res(A)\frac{1}{z}\\
 &+\varphi_0\left(\Int\sigma(A)- \frac{1}{q}\res(A\log Q)\right)-\Tr(A\Pi_Q)\\
& +\sum_{k=1}^K(-1)^k\frac{(z)^k}{k!} \\
&\times \left(\varphi_0\left( \Int\sigma(A(\log Q)^k)d\xi-\frac{1}{q(k+1)}\res(A(\log Q)^{k+1})\right)-\Tr(A\log^k Q\Pi_Q)\right)\\
&+o(z^{K}).
\end{align*}
Where $\Pi_Q$ is the projection on the kernel of $Q$.
\end{proposition} 

For  operators $A$ and $Q$ as   in the previous Proposition, 
we define a zeta function by
\begin{equation}\label{zetafunction}
\zeta(A,Q,z)=\TR(AQ^{-z}).
\end{equation}
By Corollary \ref{analyticcontin},  it is obvious that $\zeta(A,Q,z)$ is the analytic continuation of the zeta function $\Tr(AQ^{-z})$ defined by the regular trace only for $\Re(z)>>0$. 
\begin{remark}\label{holomorphicityatzero}
If $A$ is a differential operator,  the   zeta function \eqref{zetafunction} is holomorphic at $z=0$ with the value equal to 
$$\varphi_0\left(\Int\sigma(A)- \frac{1}{q}\res(A\log Q)\right)-\Tr(A\Pi_Q).$$
\end{remark}

\section{ Cauchy-Riemann operators on noncommutative tori}
  
In \cite{Quillen1985},  Quillen studies the geometry of the determinant line bundle on the space of all  Cauchy-Riemann operators on a smooth  vector bundle  on a closed Riemann surface. To investigate the same notion on noncommutative tori,  we first briefly recall some basic facts in the classical case on  how Cauchy-Riemann operators are related to Dirac operators and spectral triples. Then by analogy we  define our Cauchy-Riemann operator on $\A$, and consider the spectral triples defined by them.
 
 Let  $M$ be a compact complex manifold and $V$ be a smooth complex vector bundle on $M$. 
Let $\Omega^{p,q}(M,V)$ denote the space of $(p,q)$ forms on $M$ with coefficients in $V$.
 A $\bar\partial$-flat connection 
 on $V$ is a $\mathbb{C}$-linear map $D: \Omega^{0,0}(M,V)\to \Omega^{0,1}(M,V),$ such that  for any $f\in C^\infty(M)$ and  $u\in \Omega^{0,0}(M,V)$,
 \begin{equation}
 D(fu)=(\bar\partial f)\otimes u+ f Du,
\end{equation} 
and $D^2=0$. Here to define $D^2$,  note that 
any  $\bar\partial$-connection as above has a unique extension 
to an operator $D:\Omega^{p,q}(M,V)\to\Omega^{p,q+1}(M,V)$, defined by 
$$D(\alpha\otimes \beta)=\bar{\partial}\alpha \otimes u+(-1)^{p+q}\alpha\wedge Du, \quad \alpha\in \Omega^{p,q}(M),\,\, u\in C^\infty(V).$$  

We refer to  $\bar\partial$-flat connections as Cauchy-Riemann operators. 
 A holomorphic vector bundle $V$ has a canonical Cauchy-Riemann operator $\bar{\partial}_V:\Omega^{0}(M,V)\to \Omega^{0,1}(M,V)$, whose extension to $\Omega^{0,*}(M,V)$ forms the Dolbeault complex of $M$ with coefficients in $V$. In fact there is a one-one correspondence between Cauchy-Riemann operators on $V$ up to (gauge) equivalence, and holomorphic structures on $V$.  We denote by $\mathcal{A}$ the set of all Cauchy-Riemann operators on $V$.

Any holomorphic structure on a Hermitian vector bundle $V$ determines a unique Hermitian  connection, called the Chern connection, whose projection on $(0,1)$-forms, $\nabla^{0,1}(M, V)$, is the Cauchy-Riemann operator coming from the holomorphic structure.

Now, if $M$ is a K\"ahler manifold,  the tensor product of the Levi-Civita connection for $M$  with the Chern connection on $V$ defines a Clifford connection on the Clifford module $(\Lambda^{0,+}\oplus \Lambda^{0,-})\otimes V$  and the operator $D_0=\sqrt{2}(\bar\partial_V+\bar\partial_V^*)$ is the associated Dirac operator (see e.g. \cite{Gilkey1984}). 
Any other Dirac operator on the Clifford module $(\Lambda^{0,+}\oplus \Lambda^{0,-})\otimes V$ is of the form $D_0+A$ where $A$ is the connection one form of a Hermitian connection. 
 This connection need not be a Chern connection.  However, on a Riemann surface (with a Riemannian metric compatible with its complex structure) any Hermitian connection on a  smooth Hermitian vector bundle is the  Chern connection of a holomorphic structure on $V$. Therefore, the positive part of any Dirac operator on $(\Lambda^{0,0}\oplus \Lambda^{0,1})\otimes V$ is a Cauchy-Riemann operator,  and this gives a one to one correspondence between all Dirac operators and the  set of all Cauchy-Riemann operators.

Next we define the analogue of Cauchy-Riemann operators for the noncommutative torus. 
First, following \cite{Connes-Tred2009, Masoud-Farzad2012}, we fix a complex structure on $\A$ by a complex number $\tau$ in the upper half plane and construct the spectral triple
\begin{equation}\label{firstspec}
(\A, \mathcal{H}_{0}\oplus\mathcal{H}^{0,1},D_0=\left(\begin{array}{ll} 0 & \bar{\partial}^*\\ \bar{\partial}&0\end{array}\right)),
\end{equation}
where $\bar{\partial}:\A\to \A$ is given by $\bar\partial=\delta_1+\tau\delta_2$.  The Hilbert space $\mathcal{H}_0$ is obtained by GNS construction from $\A$ using the trace $\varphi_0$ and $\bar{\partial}^*$ is the adjoint of the operator $\bar{\partial}$.

As in the classical case, we define our Cauchy-Riemann operators on $\A$  as the positive part of twisted Dirac operators. All such operators  define  spectral triples of the  form
$$(\A, \mathcal{H}_{0}\oplus\mathcal{H}^{0,1},D_A=\left(\begin{array}{ll} 0 & \bar{\partial}^*+\alpha^*\\ \bar{\partial}+\alpha&0\end{array}\right)),$$
where $\alpha\in \A$ is the positive part of a selfadjoint element 
$$A=\left(\begin{array}{ll} 0 & \alpha^*\\ \alpha&0\end{array}\right)\in\Omega^1_{D_0}(\A).$$
We recall that $\Omega^1_{D_0}(\A)$ is the space of quantized one forms 
 consisting of the elements $\sum a_i[D_0, b_i]$ where $a_i,b_i\in\A$ \cite{Connesbook1994}.
Note that the in this case, the space  $\mathcal{A}$ of Cauchy-Riemann operators is the space of $(0,1)$-forms on $\A$.

We should mention that in the noncommutative case,  in  the work of  Chakraborty and Mathai \cite{Chakraborty-Mathai2009}   a general family of spectral triples is considered and, under suitable regularity conditions, a   determinant line bundle is defined for such families.  
The  curvature of the determinant line bundle  however is not computed  and that is the main object of study in the present paper, as well as in \cite{Quillen1985}.   

 \section{The curvature of the determinant line bundle for $\A$}\label{sec:computation}
For any $\alpha\in\mathcal{A}$,  the Cauchy-Riemann operator $$\bar{\partial}_{\alpha}=\bar{\partial}+\alpha:\mathcal{H}_0\to\mathcal{H}^{0,1}$$ is a Fredholm operator.
We  pull back the determinant line bundle DET on  the space of Fredholm operators  ${\rm Fred}(\mathcal{H}_0,\mathcal{H}^{0,1}),$ to get a line bundle $\mathcal{L}$ on $\mathcal{A}$. Following Quillen \cite{Quillen1985},  we define a Hermitian metric on $\mathcal{L}$ and compute its curvature in this section.  Let us define a metric on  the fiber 
$$\mathcal{L}_{\alpha} = \Lambda^{max} ({\rm ker} \,\bar{\partial}_{\alpha})^*\otimes \Lambda^{max}({\rm ker} \,\bar{\partial}_{\alpha}^*).$$ 
as the product of the induced metrics on $\Lambda^{max} ({\rm ker} \,\bar{\partial}_{\alpha})^* $,  $\Lambda^{max}({\rm ker} \,\bar{\partial}_{\alpha}^*)$, 
 with the zeta regularized determinant $e^{-\zeta'_{\Delta_{\alpha}}(0)}$.
Here we define   the Laplacian  as $\Delta_{\alpha}=\bar{\partial}_{\alpha}^*\bar{\partial}_{\alpha}:\mathcal{H}_0\to \mathcal{H}_0$,  and its zeta function by
\begin{align*}
&\zeta(z)=\TR(\Delta_{\alpha}^{-z}).
\end{align*}
It is a meromorphic function and by Remark \ref{holomorphicityatzero} it is regular at $z=0$ .
Similar proof as in \cite{Quillen1985} shows that this defines a smooth Hermitian metric on $\mathcal{L}$. 

On the open set of invertible  operators  each fiber of $\mathcal{L}$  is canonically isomorphic to $\mathbb{C}$ and the  nonzero holomorphic section $\sigma=1$  gives a trivialization. Also, according to the definition of the Hermitian metric, the norm of this section is given by 
\begin{equation}
\|\sigma\|^2=e^{-\zeta'_{\Delta_\alpha}(0)}.
\end{equation} 

\subsection{Variations of LogDet and curvature form}
We begin by explaining the motivation behind the computations of Quillen in  \cite{Quillen1985}. Recall that    a holomorphic line bundle equipped with a  Hermitian  inner product has a canonical  connection compatible with the two structures. This is also known as the Chern connection. The curvature form of this connection is computed by  $\bar{\partial} \partial \log\|\sigma\|^2,$ 
where $\sigma$  is any non-zero local holomorphic section.

In our case we will proceed by analogy and  compute the second variation 
$ \bar{\partial} \partial \log \|\sigma\|^2$
on the open set of invertible index zero Cauchy-Riemann operators. Let us  consider a holomorphic family of invertible index zero Cauchy-Riemann operators $D_w=\bar{\partial}+\alpha_w$, where $\alpha_w$ depends holomorphically on the complex variable $w$ and compute 
$$ \delta_{\bar{w}} \delta_w \zeta'_\Delta(0).$$ 
One has the following first variational formula,
\begin{equation*}
\delta_w\zeta(z)=\delta_w\TR(\Delta^{-z})
=\TR(\delta_w\Delta^{-z})
=-z\TR(\delta_w\Delta\Delta^{-z-1}),
\end{equation*} 
where in the second equality we were able to  change the order of $\delta_w$ and $\TR$ because of the uniformity condition in the definition of holomorphic families (cf. \cite{Paycha-Rosenberg2006}).  

Note that, although $\TR(\Delta^{-z})$ is regular at $z=0$,  
$\TR(\delta_w\Delta\Delta^{-z-1})$ might have a pole at $z=0$ since $\delta_w\Delta\Delta^{-z-1}|_{z=0}=\delta_w\Delta\Delta^{-1}$ is not a differential operator any more and may have non-zero residue. 
Around $z=0$ one has the following Laurent expansion:
$$-z\TR(\delta_w\Delta\Delta^{-z-1})=-z(\frac{a_{-1}}{z}+a_0+a_1 z+\cdots).$$
Hence, 
$$\left.\delta_w\zeta(z)\right|_{z=0}=-a_{-1},\qquad \left.\frac{d}{d z}\delta_w\zeta(z)\right|_{z=0}=-a_{0}.$$
Using Proposition \ref{Laurentat0} we have 
$$\delta_w\zeta'(0)=\left.\frac{d}{d z}\delta_w\zeta(z)\right|_{z=0}=-\varphi_0 \left(\Int\sigma(\delta_w\Delta\Delta^{-1})-\frac{1}{2}{\rm res}_x(\delta_w\Delta\Delta^{-1}\log\Delta)\right).$$
To compute the right hand side of the above equality,
 we need to note that since $D_w$ depends holomorphically on $w$, $\delta_wD^*=0$ and hence
$$\delta_w\Delta=\delta_w D^*D+ D^*\delta_w D=D^*\delta_w D.$$ 
Since $\delta_wD$ is a zero order differential operator, we have
\begin{align*}
\delta_w\zeta'(0)&=-\varphi_0\left(\Int\sigma(D^*\delta_w D\Delta^{-1})-\frac{1}{2}{\rm res}(D^*\delta_w D\Delta^{-1}\log\Delta)\right)\\
&=-\varphi_0\left(\Int\sigma(\delta_w D\Delta^{-1}D^*)-\frac{1}{2}{\rm res}(\delta_w D\log\Delta\,\Delta^{-1}D^*)\right)\\
&=-\varphi_0\left(\delta_w D\left(\Int\sigma(D^{-1})-\frac{1}{2}{\rm res}(\log\Delta\,D^{-1})\right)\right)\\
&=-\varphi_0\left(\delta_w D\, J\right),
\end{align*}   
where 
$$J=\Int\sigma(D^{-1})-\frac{1}{2}{\rm res}(\log(\Delta)D^{-1}).$$ 
The reader can compare this to the term $J$ in Quillen's computations \cite{Quillen1985}.

Now we compute the second variation 
$\delta_{\bar w}\delta_{w}\zeta'(0).$
Since $D_w$ is holomorphic we have
\begin{align*}
\delta_{\bar w}\delta_{w}\zeta'(0)&=-\varphi_0\left(\delta_w D \delta_{\bar w}J\right).
\end{align*}
Next we compute the variation  $\delta_{\bar{w}}J$. Note that since $D_w$ is invertible, $D_w^{-1}$ is also holomorphic and hence $\delta_{\bar{w}}\Int\sigma(D^{-1})=0$. Therefore 
\begin{align*}
 \delta_{\bar w}J&=\delta_{\bar{w}}\left( \Int\sigma(D^{-1})-\frac{1}{2}{\rm res}(\log\Delta\, D^{-1})\right)
 =-\frac{1}{2}\delta_{\bar{w}}{\rm res}(\log\Delta\, D^{-1}).
\end{align*}
 
Thus, we have shown that
 \begin{lemma}\label{secondvarnotsim} 
 For the holomorphic family of Cauchy-Riemann operators $D_w$, the second variation of $\zeta'(0)$ reads:
 $$\delta_{\bar w}\delta_{w}\zeta'(0)=\frac{1}{2}\varphi_0\left(\delta_w D\delta_{\bar{w}}{\rm res}(\log\Delta\,D^{-1})\right).$$
 \qed
 \end{lemma}
Our next goal is to compute $\delta_{\bar{w}}{\res}(\log\Delta\,D^{-1})$. 
This combined with the above lemma shows that  the curvature  form of the determinant line bundle equals the K\"ahler form on the space of connections. 

 \begin{lemma}\label{secondvarsim} 
 With above definitions and notations,  we have
\begin{align*}
 \sigma_{-2,0}(\log\Delta\, D^{-1})&=\frac{(\alpha+\alpha^*)\xi_1+(\bar\tau\alpha+\tau\alpha^*)\xi_2}{(\xi_1^2+2\Re(\tau) \xi_1\xi_2+ |\tau|^2\xi_2^2)(\xi_1+\tau\xi_2) }\\ \\
 &-\log \left( \frac{\xi_1^2+2\Re(\tau) \xi_1\xi_2+ |\tau|^2\xi_2^2}{|\xi|^2}\right) \frac{\alpha}{\xi_1+\tau\xi_2},
 \end{align*}
 and
 $$\delta_{\bar{w}}{\rm res}(\log(\Delta) D^{-1})=\frac{1}{2\pi\Im(\tau)}(\delta_{{w}}D)^*.$$
  \end{lemma}
  \begin{proof}
By writing down the homogeneous terms in the expansion of $\sigma_{\bullet,0}(\log\Delta)$ and $\sigma(D^{-1})$ and using the product formula of the symbols we see that
$$\sigma_{-2,0}(\log\Delta D^{-1})\sim\sigma_{-1,0}(\log\Delta)\sigma_{-1}(D^{-1})+\sigma_{0,0}(\log\Delta)\sigma_{-2}(D^{-1}).$$
  
 Starting with the symbol of $\Delta$, we have
  $$\sigma(\Delta)=\xi_1^2+2\Re(\tau) \xi_1\xi_2+ |\tau|^2\xi_2^2+(\alpha+\alpha^*)\xi_1+(\bar\tau\alpha+\tau\alpha^*)\xi_2+\bar{\partial}^*(\alpha).$$
  Then, the homogeneous parts of  $\sigma((\lambda-\Delta)^{-1})=\sum_{j} b_{-2-j}$ is given by the following recursive formula 
 \begin{align*}
 b_{-2}&=(\lambda-\sigma_{2}(\Delta))^{-1},\\
 b_{-2-j}&=-b_{-2}\sum_{k+l+|\gamma|=j, \, l<j}\partial^\gamma\sigma_{2-k}(\Delta)\delta^\gamma b_{-2-l}/\gamma!,
 \end{align*}
which gives us
 $$b_{-2}=\frac{1}{\lambda-(\xi_1^2+2\Re(\tau) \xi_1\xi_2+ |\tau|^2\xi_2^2)},$$
and
 $$b_{-3}=\frac{1}{(\lambda-(\xi_1^2+2\Re(\tau) \xi_1\xi_2+ |\tau|^2\xi_2^2))^2}\left((\alpha+\alpha^*)\xi_1+(\bar\tau\alpha+\tau\alpha^*)\xi_2\right).$$
 
Also, $\Delta^z$ is a classical operator defined by 
 $$\Delta^z=\frac{1}{2\pi i}\int_C \lambda^z(\lambda-\Delta)^{-1}d\lambda,$$
 with the homogeneous parts of the symbol given by 
 $$b(z)_{2z-j}:=\sigma_{2z-j}(\Delta^z)=\frac{1}{2\pi i}\int_C \lambda^zb_{-2-j}d\lambda.$$
 Hence we have 
\begin{align*}
b{(z)}_{2z}&=\frac{1}{2\pi i}\int_C \lambda^z\frac{1}{\lambda-(\xi_1^2+2\Re(\tau) \xi_1\xi_2+ |\tau|^2\xi_2^2)}d\lambda\\ \\
&=(\xi_1^2+2\Re(\tau) \xi_1\xi_2+ |\tau|^2\xi_2^2)^{z}
\end{align*}
\begin{align*}
b{(z)}_{2z-1}&=\frac{1}{2\pi i}\int_C \lambda^z\frac{\left((\alpha+\alpha^*)\xi_1+(\bar\tau\alpha+\tau\alpha^*)\xi_2\right)}{(\lambda-(\xi_1^2+2\Re(\tau) \xi_1\xi_2+ |\tau|^2\xi_2^2))^2}d\lambda \\ \\
&=z(\xi_1^2+2\Re(\tau) \xi_1\xi_2+ |\tau|^2\xi_2^2))^{z-1} \left((\alpha+\alpha^*)\xi_1+(\bar\tau\alpha+\tau\alpha^*)\xi_2\right).
\end{align*}

Using \eqref{cllog} and what we have computed up to here, it is clear that 
\begin{align*}
 \sigma_{0,0}(\log\Delta)(\xi)&= \sigma_{2}(\Delta)|\xi|^{-2}\left.\frac{d}{d z}\right|_{z=0} b{(z-1)}_{2z-2}\left(\xi/|\xi|\right)\\ \\
 &=\sigma_{2}(\Delta)|\xi|^{-2}\left.\frac{d}{d z}\right|_{z=0}( (\xi_1^2+2\Re(\tau) \xi_1\xi_2+ |\tau|^2\xi_2^2)/|\xi|^2)^{z-1}\\ \\
 &=\log ( (\xi_1^2+2\Re(\tau) \xi_1\xi_2+ |\tau|^2\xi_2^2)/|\xi|^2).
 \end{align*}
 Note that the above term is homogeneous of order zero in $\xi$.
 \begin{align*}
 & \sigma_{-1,0}(\log\Delta)(\xi)\\ \\
  &= \sum_{i+j+|\alpha|=1}\frac{1}{\alpha !}\partial^\alpha \sigma_{2-i}(\Delta)\delta^\alpha|\xi|^{-2-j}\left.\frac{d}{d z}\right|_{z=0} b{(z-1)}_{2z-2-j}\left(\xi/|\xi| \right)\\ \\
  &=  \sigma_{2}(\Delta)|\xi|^{-3}\left.\frac{d}{d z}\right|_{z=0}  b{(z-1)}_{2z-3}\left(\xi/|\xi|\right)\\
  &+ \sigma_{1}(\Delta)|\xi|^{-2}\left.\frac{d}{d z}\right|_{z=0}  b{(z-1)}_{2z-2}\left(\xi/|\xi|\right)\\ \\
    &=\frac{1-\log (\xi_1^2+2\Re(\tau) \xi_1\xi_2+ |\tau|^2\xi_2^2)/|\xi|^2)}{ (\xi_1^2+2\Re(\tau) \xi_1\xi_2+ |\tau|^2\xi_2^2)} \left[(\alpha+\alpha^*)\xi_1+(\bar\tau\alpha+\tau\alpha^*)\xi_2\right]\\
   & +\frac{ \log  (\xi_1^2+2\Re(\tau) \xi_1\xi_2+ |\tau|^2\xi_2^2)/|\xi|^2)}{\xi_1^2+2\Re(\tau) \xi_1\xi_2+ |\tau|^2\xi_2^2} \left[(\alpha+\alpha^*)\xi_1+(\bar\tau\alpha+\tau\alpha^*)\xi_2\right]
  \\ \\
  &= (\xi_1^2+2\Re(\tau) \xi_1\xi_2+ |\tau|^2\xi_2^2)^{-1} \left[(\alpha+\alpha^*)\xi_1+(\bar\tau\alpha+\tau\alpha^*)\xi_2\right].
  \end{align*}

 Next we compute the  symbol of $D^{-1}$. The  symbol of $D$ reads
 $$\sigma(D)={\xi_1+\tau \xi_2}+\alpha.$$ 
 We need to compute the homogeneous parts of  order -1 and -2 of $D^{-1}$. By using recursive formula for the symbol of the inverse we get:
 \begin{align*}
 \sigma_{-1}(D^{-1})&=\sigma_1(D)^{-1}=(\xi_1+\tau\xi_2)^{-1}\\
 \sigma_{-2}(D^{-1})&=- \sigma_{-1}(D^{-1})\sum_{k+|\gamma|=1}\partial^\gamma\sigma_{1-k}(D)\delta^\gamma  \sigma_{-1}(D^{-1})/\gamma!\\
 &=- \sigma_{-1}(D^{-1})^2\sigma_{0}(D)\\
  &=- (\xi_1+\tau\xi_2)^{-2}\alpha.
 \end{align*}
 Finally, we have
 \begin{align*}
 \sigma_{-2,0}(\log\Delta\, D^{-1})&=\sigma_{-1,0}(\log\Delta)\sigma_{-1}(D^{-1})+\sigma_{0,0}(\log\Delta)\sigma_{-2}(D^{-1})\\
 &=(\xi_1^2+2\Re(\tau) \xi_1\xi_2+ |\tau|^2\xi_2^2)^{-1}(\xi_1+\tau\xi_2)^{-1} \left[(\alpha+\alpha^*)\xi_1+(\bar\tau\alpha+\tau\alpha^*)\xi_2\right]\\
 &-\log ( (\xi_1^2+2\Re(\tau) \xi_1\xi_2+ |\tau|^2\xi_2^2)/|\xi|^2) (\xi_1+\tau\xi_2)^{-2}\alpha.
 \end{align*}
 Therefore, we compute the variation:
 \begin{align}\label{variation}
\delta_{\bar{w}} \sigma_{-2,0}(\log\Delta\,  D^{-1})&=(\xi_1^2+2\Re(\tau) \xi_1\xi_2+ |\tau|^2\xi_2^2)^{-1} \left[(\delta_{\bar w}\alpha^*)\xi_1+(\tau\delta_{\bar w}\alpha^*)\xi_2\right](\xi_1+\tau\xi_2)^{-1}\nonumber\\
&=(\xi_1^2+2\Re(\tau) \xi_1\xi_2+ |\tau|^2\xi_2^2)^{-1}(\delta_{\bar w}\alpha^*)\nonumber\\
&=(\xi_1^2+2\Re(\tau) \xi_1\xi_2+ |\tau|^2\xi_2^2)^{-1}(\delta_{{w}}D)^*.
\end{align}
In order to compute the variation of the residue density, we need to integrate (\ref{variation}) with respect to $\xi$ variable:
$$\delta_{\bar{w}}{\rm res}(\log(\Delta) D^{-1})=\int_{|\xi|=1}(\xi_1^2+2\Re(\tau) \xi_1\xi_2+ |\tau|^2\xi_2^2)^{-1}(\delta_{{w}}D)^*d\xi=\frac{1}{2\pi\Im(\tau)}(\delta_{{w}}D)^*.$$
Note that we have used the normalized Lebesgue measure in the last integral (see \eqref{pseudodef}).
\end{proof}

We record the main result of this paper in the following theorem. It computes the curvature of the determinant line bundle in terms of the natural K\"{a}hler form on the space of connections.
\begin{theorem}
The curvature of the determinant line bundle for the  noncommutative two torus is given by
\begin{equation}
\delta_{\bar w}\delta_{w}\zeta'(0)=\frac{1}{4\pi\Im(\tau)}\varphi_0\left(\delta_w D(\delta_w D)^*\right).
\end{equation}\qed
\end{theorem}

\begin{remark}
In order to recover the classical result of Quillen for $\theta=0$, we have to take into account the change of the volume form due to a change of the metric. This means we have to  multiply the  above  result by $\Im(\tau)$.
\end{remark}

\def\polhk#1{\setbox0=\hbox{#1}{\ooalign{\hidewidth
  \lower1.5ex\hbox{`}\hidewidth\crcr\unhbox0}}} \def\cprime{$'$}

\end{document}